\newif\ifpictures
\picturestrue

\documentclass[a4paper,12pt]{amsart}

\usepackage[latin1]{inputenc}
\usepackage{a4wide}
\usepackage{setspace}
\usepackage{relsize} 
\usepackage{amsopn} 
\usepackage{xspace} 
\usepackage{amsthm}
\usepackage{supertabular}
\usepackage{amssymb}
\usepackage{amsfonts}
\usepackage{stmaryrd}
\usepackage{xcolor}
\usepackage[dvips]{epsfig}
\usepackage{float}
\usepackage[arrow, matrix, curve]{xy}
\usepackage[vlined,boxed]{algorithm2e}

\newcommand{\id}{\textrm{id}}

\newcommand{\spn}{\textrm{span}}

\DeclareMathOperator{\conv}{conv }

\DeclareMathOperator{\ttint}{int}
\DeclareMathOperator{\pyr}{pyr }
\DeclareMathOperator{\bpyr}{bipyr }

\newcommand{\py}[1]{\pyr_{#1}}

\DeclareMathOperator{\PS}{\mathcal{S}}
\DeclareMathOperator{\PM}{\mathcal{M}}
\DeclareMathOperator{\PW}{\mathcal{W}}

\newcommand{\eps}{\varepsilon}

\newcommand{\lam}{\lambda}

\newcommand{\sig}{\sigma}

\newcommand{\Sig}{\Sigma}

\newcommand{\tri}{\triangle}

\newcommand{\lf}{\left}
\newcommand{\ri}{\right}
\newcommand{\ra}{\rightarrow}
\newcommand{\Ra}{\Rightarrow}

\newcommand{\Lera}{\Leftrightarrow}

\newcommand{\w}{\wedge}

\newcommand{\wh}{\widehat}

\newcommand{\bs}{\backslash}

\newcommand\fN{{\ensuremath{\mathbb{N}}}\xspace}

\newcommand\fR{{\ensuremath{\mathbb{R}}}\xspace}

\newcommand\cA{{\ensuremath{\mathcal{A}}}\xspace}

\newcommand\cF{{\ensuremath{\mathcal{F}}}\xspace}

\newcommand\cH{{\ensuremath{\mathcal{H}}}\xspace}

\newcommand\cM{{\ensuremath{\mathcal{M}}}\xspace}

\newcommand\cO{{\ensuremath{\mathcal{O}}}\xspace}

\newcommand\cS{{\ensuremath{\mathcal{S}}}\xspace}

\newcommand\cV{{\ensuremath{\mathcal{V}}}\xspace}

\newif\ifpictures
\picturestrue

\usepackage{amssymb,amsmath}
\usepackage{hyperref}
\usepackage[dvips]{graphicx}

\ifpictures
\usepackage{psfrag}
\fi

\numberwithin{equation}{section}

\newtheoremstyle{MatheStandard}
{}{10pt} 
{} 
{-8pt} 
{\bfseries}{:} 
{\newline} 
{} 

\newtheorem{satz}{Theorem}[section]
\newtheorem{lem}[satz]{Lemma}
\newtheorem{cor}[satz]{Corollary}
\newtheorem{defn}[satz]{Definition}
\newtheorem{exa}[satz]{Example}

\newtheorem{thm}[satz]{Theorem}
\newtheorem{prop}[satz]{Proposition}
\newtheorem{rem}[satz]{Remark}

\newtheorem*{satz*}{Theorem}
\newtheorem*{lem*}{Lemma}
\newtheorem*{cor*}{Corollary}
\newtheorem*{defn*}{Definition}
\newtheorem*{exa*}{Example}
\newtheorem*{thm*}{Theorem}
\newtheorem*{prop*}{Proposition}
\newtheorem*{rem*}{Remark}
\newtheorem*{conj*}{conj}

\newtheoremstyle{MatheExtra}
{}{10pt} 
{} 
{-8pt} 
{\itshape}{:} 
{\newline} 
{} 

\usepackage{booktabs}

\title{Polytopes with special simplices}
\author{Timo de Wolff}
\date{\today}

\begin{document}
\maketitle

\begin{abstract}
For a polytope $P$ a simplex $\Sig$ with vertex set $\cV(\Sig)$ is called a \textit{special simplex} if every facet of $P$ contains all but exactly one vertex of $\Sig$.

For such polytopes $P$ with face complex $\cF(P)$ containing a special simplex the subcomplex $\cF(P) \bs \cV(\Sig)$ of all faces not containing vertices of $\Sig$ is the boundary of a polytope $Q$ --- the \textit{basis polytope} of $P$. If additionally the dimension of the affine basis space of $\cF(P) \bs \cV(\Sig)$ equals $\dim(Q)$, we call $P$ \textit{meek}; otherwise we call $P$ \textit{wild}.

We give a full combinatorial classification and techniques for geometric construction of the class of meek polytopes with special simplices. We show that every wild polytope $P'$ with special simplex can be constructed out of a particular meek one $P$ by intersecting $P$ with particular hyperplanes. It is non--trivial to find all these hyperplanes for an arbitrary basis polytope; we give an exact description for 2--basis polytopes.  Furthermore we show that the $f$--vector of each wild polytope with special simplex is componentwise bounded above by the $f$--vector of a particular meek one which can be computed explicitly. Finally, we discuss the $n$--cube as a non--trivial example of a wild polytope with special simplex and prove that its basis polytope is the zonotope given by the Minkowski sum of the $(n-1)$--cube and vector $(1,\ldots,1)$.

Polytopes with special simplex have applications on Ehrhart theory, toric rings and were just used by Francisco Santos to construct a counter--example disproving the Hirsch conjecture.
\end{abstract}

\section{Introduction}
\label{SecIntroduction}
Studying the combinatorics of polytopes or of classes of polytopes is a main topic in polytope theory with many applications (as a general reference see, e.g., \cite{Gruenb1,Zieg1}). In this paper we study \emph{polytopes
with a special simplex} which have been introduced by C. Athanasiadis in his work on Ehrhart series, $h$-vectors of polytopes and Gorenstein toric rings \cite{Atha1}.\\

For a polytope $P$, let $\cV(P)$ denote the vertex set of $P$, $\cF(P)$ denote the face complex and $\cF(\partial P)$ denote the boundary complex of $P$ (see e.g. \cite[p. 129]{Zieg1}). Similarly, let $\cV(\cF(P))$ denote the
vertex set of any polytopal complex $\cF(P)$. For every set $\sig \subseteq \cV(\cF(P))$ consisting of vertices of a polyhedral complex $\cF(P)$, we denote the maximal subcomplex of $\cF(P)$ which does not contain vertices in $\sig$ as $\cF(P) \bs \sig$.

\begin{defn}[\cite{Atha1}]
Let $P$ be an $n$--polytope in $\fR^q$ and let $\Sig$ be a simplex spanned by $m+1$ vertices of $P$. Then we call $\Sig$ a \emph{special simplex} in $P$, if every facet of $P$ contains exactly $m$ vertices of $\Sig$ (i.e. every facet contains all but one vertex of $\Sig$). We define
\begin{eqnarray*}
	\PS_{(n,m)}	& := & \{P \ | \ P \text{ is a polytope with some special simplex } \Sig, \dim(P) = n, \dim(\Sig) = m\}
\end{eqnarray*}

\label{DefSpezSimpl}
\end{defn}

In \cite{Atha1} Athanasiadis proves using the results of Stanley (\cite{Stan1}) as well as Reiner and Welker (\cite{ReiWel1}) that for a compressed $n$--polytope with special $m$--simplex the Ehrhart series can be written as
\begin{eqnarray*}
	\sum_{r \geq 0} i(P,r)t^r	& =	& \frac{h(t)}{(1-t)^{n-1}}
\end{eqnarray*}
with $h(t)$ being the $h$--polynomial of the boundary complex of a simplicial polytope in dimension $n - m$. Therefore, $h(t)$ is in particular symmetric and unimodal (i.e.: $h_0 \leq \ldots \leq h_{\lfloor n / 2 \rfloor}$).

The connection to toric rings shown by Athanasiadis, Ohsugi and Hibi in \cite{Atha1}, \cite{OhHi2} and \cite{OhHi3} is that a compressed polytope $P$ contains a special simplex if and only if the toric ring $K[P]$ of $P$ is \textit{Gorenstein}.

Independently from Ohsugi and Hibi, Bruns and Römer generalized Athanasiadis' theorem to $h$--vectors of Gorenstein polytopes (cf. \cite{BruRoe1}). See also e.g. \cite{ConHosTho1} and \cite{BatNill1}.

In \cite{Santos1} Santos disproves the Hirsch conjecture by constructing a counter example in dimension 43. The key part for the construction of this counter example is his proof for the existence of a 5--dimensional polytope with special 1--simplex\footnote{He calls a polytope with special 1--simplex a \textit{spindle}.} which has a length of 6. The length is here defined as the number of edges one has to pass to get from one vertex of the special simplex to the other one.\\

It is known that \textit{Birkhoff polytopes} and \textit{order polytopes} over graded posets come always with a special simplex. One can find some other examples in low dimensions quite easily; furthermore a connection between the \textit{reverse lexicographic triangulation} of such polytopes and \textit{simplicial joins} can be established. However, so far there has not been a systematic classification of this class of polytopes.\\

In this article we give a full classification of polytopes with special simplex in the following way:

Firstly, we define \textit{basis polytopes}. Note that we call two polytopal (resp. simplicial / polyhedral) complexes (and therefore in particular: two polytopes) \emph{isomorphic} if their face lattices are isomorphic.
\begin{defn}
Let $P$ be an $n$--polytope in $\fR^n$ with face complex $\cF(P)$ and special $m$--simplex $\Sig$. We define the \emph{basis polytope} $Q$ as the image of $P$ under the projection $\pi: \fR^n \ra \fR^{n-m}$ of the linear subspace, which is parallel to the affine basis space of $\Sig$, to the origin.
\label{DefBasispolytop}
\end{defn}

We use the notation
\begin{eqnarray*}
	  \PS_{(n,m)}(Q)	& := & \lf\{P \ \middle| \ \begin{array}{l}P \text{ is polytope with some special simplex } 					\Sig, \dim(P) = n, \\
	   					\dim(\Sig) = m, Q \text{ is basis polytope of } P
	                	       \end{array}\ri\}
\end{eqnarray*}

In Proposition \ref{SatzAequivSpezSimp} we deduce amongst other things by using results of Athanasiadis (\cite{Atha1}) as well as Reiner and Welker (\cite{ReiWel1}) that 
\begin{eqnarray*}
	\cF(\partial Q)	& =	& \cF(P) \bs \cV(\Sig)
\end{eqnarray*}
with $\cF(P) \bs \cV(\Sig)$ being the subcomplex of the faces of $P$ which does not contain any vertex of $\cV(\Sig)$. Note that it is not obvious that $\cF(P) \bs \cV(\Sig)$ is the boundary complex of a polytope.\\

Basis polytopes and the upper fact induce a natural distinction between \textit{meek} and \textit{wild} polytopes with special simplex:
\begin{defn}
Let $P$ be an $n$--polytope with special simplex $\Sig$ and basis polytope $Q$. Let $\cA$ be the affine basis space of the polytopal complex $\cF(P) \bs \cV(\Sig)$. We call $P$
\begin{enumerate}
 \item[(a)] \emph{meek}, if and only if $\dim Q = \dim \cA$.
 \item[(b)] \emph{wild}, if and only if $\dim Q < \dim \cA$ and the combinatorial structure of $P$ is different from the meek polytope $\conv (\cV(Q) \cup \cV(\Sig))$\footnote{Background for this condition is that it might be possible to shift some vertices of $Q$ out of the affine basis space without changing the combinatorial structure of $\conv (\cV(Q) \cup \cV(\Sig))$. Since we are only interested in combinatorial classification and geometric construction we want to omit these special cases.}. We use the notation
 \begin{eqnarray*}
	\PM_{(n,m)}(Q)	& := & \lf\{P \ \middle| \ \begin{array}{l} P \text{ is meek polytope with some special 						simplex } \Sig, \dim(P) = n, \\
	  					\dim(\Sig) = m, Q \text{ is basis polytope of } P
	                		\end{array}\ri\} \\
	\PW_{(n,m)}(Q)	& := & \lf\{P \ \middle| \ \begin{array}{l} P \text{ is wild polytope with some special 						simplex } \Sig, \dim(P) = n, \\
	  					\dim(\Sig) = m, Q \text{ is basis polytope of } P
	                		\end{array}\ri\}
\end{eqnarray*}
\end{enumerate}
\label{DefZahmWild}
\end{defn}

With these tools we firstly get the following theorem which classifies meek polytopes with special simplex:
\begin{thm}
 Let $P$ be a meek $n$--polytope with special $m$--simplex $\Sig$. Then $P$ has an $(n - m)$--dimensional basis polytope $Q$ (given by Definition \ref{DefBasispolytop}) such that
\begin{eqnarray*}&&
	\begin{array}{cccccccc}
		\exists \, i \in \fN, j \in \fN_{>0} :	& i + j	& =	& m	& \text{ and }	& P 	& =	& \py{i}\lf(\Sig_j \oplus Q\ri).
	\end{array}
\end{eqnarray*}
\label{ThmKlassSpezSimpRough}
\end{thm}
In this context ``$\py{i}$'' denotes the $i$--th \textit{pyramid} and ``$\Sig_j \oplus Q$'' the \textit{direct sum} of a $j$--simplex $\Sig_j$ and $Q$. Recall that the direct sum of two polytopes $P_1 \in \fR^n,P_2 \in \fR^m$ is the polytope given by
\begin{eqnarray*}
	P_1 \oplus P_2	& :=	& \lf\{\lam_1 \binom{v}{0} + \lam_2 \binom{0}{w} \in \fR^{n+m} : v \in P_1; w \in P_2; \lam_1,\lam_2 \in \fR_{\geq 0}; \lam_1 + \lam_2 \leq 1; \ri\}.
\end{eqnarray*}
Its boundary complex coincides with the polytopal join (cf. Definition \ref{DefSimplVerein}) of the boundary complex of the basis polytope and the boundary complex of a $j$--simplex.\\

With this notation Theorem \ref{ThmKlassSpezSimpRough} proves that for a fixed basis polytope $Q$ and fixed $n,m$ there are (up to labeling) exactly $m+1$ non--isomorphic meek polytopes with special simplex, i.e.:
\begin{eqnarray*}
	\# \PM_{(n,m)}(Q)	& =	& m+1	
\end{eqnarray*}
and in particular for every $n, m$ with $m \in \fN_{>0}, n \in \fN_{\geq m}$ there exists an $n$--polytope with special $m$--simplex. For arbitrary $n$ or $Q$ there exist infinitely many polytopes with special simplex if $n - m \geq 3$. Since the $f$--vectors of pyramids and direct sums are known the theorem yields the combinatorial structure of the class $\PM_{(n,m)}(Q)$ and we can therefore decide whether a polytope is meek and contains a special simplex just by knowing its $f$--vector. The theorem explains furthermore how polytopes in $\PM_{(n,m)}(Q)$ can be constructed in a combinatorial \textit{and} in a geometric way and hence classifies $\PM_{(n,m)}(Q)$ completely.\\

For $\PW_{(n,m)}(Q)$ we are unable to decide wether a polytope belongs to this class in a pure combinatorial way just by investigating its $f$--vector since the combinatorial structure of particular wild polytopes depends crucially on their basis polytopes (cf. Section \ref{SubSecWildePolytope}). 

We are able to give a classification in the following sense: We prove that every wild polytope with basis polytope $Q$ can be constructed by intersecting $\Sig \oplus Q$ with particular (via the vertices combinatorially defined) hyperplanes satisfying conditions which we present in Theorem \ref{ThmWildPolyCharac}. Therefore we present a full classification of all polytopes with special simplex. Theorem \ref{ThmWildPolyCharac} is constructive if one knows all particular hyperplanes satisfying the conditions of the Theorem for a specific basis polytope. In general it seems to be a hard task to find these hyperplanes; we present an explicit description for all polytopes with special simplex and a basis polytope of dimension 2 (Corollary \ref{CorWildPolyBasisPolyDim2}). Theorem \ref{ThmWildPolyCharac} yields furthermore an upper bound for the $f$--vector of wild polytopes with special simplex (Corollary \ref{CorBoundaries}):

Let $P \in \PW_{(n,k)}(Q)$ with special simplex $\Sig$ and $n \geq 3$ (for $n < 2$ we have $\PW_{(n,k)}(Q) = 0$). Then we have for its $f$--vector $f_P$ in particular:
\begin{eqnarray*}
		f_{P}^{(0)}	& =	& f_{(\Sig \oplus Q)}^{(0)} \\
		f_{P}^{(i)}	& <	& f_{(\Sig \oplus Q)}^{(i)} \quad \text{ for } i \in \{n-2,n-1\} \\
		f_{P}^{(i)}	& \leq	& f_{(\Sig \oplus Q)}^{(i)} \quad \text{ for } i \in \{1,\ldots,n-3\}.
\end{eqnarray*}
Here $f_{(\Sig \oplus Q)}$ denotes the $f$--vector of the direct sum which is explicitly computable.

Finally we discuss the 4--cube as an example for a wild polytope with special simplex which is non--trivial since its basis polytope has dimension 3 and thus is not covered by Corollary \ref{CorWildPolyBasisPolyDim2}. We prove in particular that the basis polytope of the $n$--cube is the zonotope given by the $(n-1)$ cube and the vector $(1,\ldots,1)$ (Proposition \ref{PropBasispolyCube}).

\subsection*{Acknowledgements} I would like to thank Hartwig Bosse, Christian Haase, Cordian Riener, Raman Sanyal, Reinhard Steffens and Theresa Szczepanski for their helpful comments. My very special thanks go to Thorsten Theobald who introduced me to the topic, supervised the master thesis (Diplomarbeit) the first part of this article is originally based on and strongly supported me in all intents and purposes from the first to the last moment.

\section{Preliminaries and examples}
\label{SecGrundlagen}
\subsection{Preliminaries}

We assume that the reader is familiar with standard geometric structures e.g.: polyhedral / simplicial complexes or $f$--vectors (cf. e.g. \cite{Zieg1}). In this section we introduce the further structures which are needed to achieve our results.\\

For our efforts we need a particular triangulation on polytopal complexes --- the \textit{reverse lexicographic triangulation} (cf. eg. \cite{Lee1}, \cite[p. 165]{Atha1}):

\begin{defn}
Let $\cF$ be a polytopal complex with vertexset  $\cV(\cF) :=\{v_1,\ldots,v_p\}$; let $\tau := (v_1,\ldots,v_p)$ be a total ordering on $\cV(\cF)$. Then we define the \emph{reverse lexicographic triangulation} (short: RLT) by:
\begin{eqnarray*}
	\tri_{\tau}(\emptyset)	& :=	& \emptyset \textrm{ and } \\
	\tri_{\tau}(\cF) 	& :=	& \tri_{\tau}(\cF \bs v_p) \cup \bigcup_{F \in M_{v_p}} \lf\{\conv(\{v_p\} \cup G) \ | \ G \in \tri_{\tau}(\cF(F)) \cup \{\emptyset\} \ri\},	
\end{eqnarray*}
with $M_{v_p}$ containing exactly the facets not containing $v_p$ of all maximal faces of $\cF$ containing $v_p$.
$\tri_{\tau}(\cF \bs v_p)$ and $\tri_{\tau}(\cF(F))$ are defined with respect to subsets of vertices of polytopal complexes and the ordering on the particular subset which is induced by $\tau$.
\label{DefRueckLexTri}
\end{defn}

Note that there are additional definitions which are equivalent to Definition \ref{DefRueckLexTri} (cf. e.g. \cite[p. 333]{Stan1} and \cite[p. 1253 et seq.]{JoZi1}; in the latter the ordering of the vertices is transposed). These triangulations coincide with the successive pulling of vertices into the negative halfspaces of the supporting hyperplanes containing the particular vertices. (cf. \cite[p. 78 et seq.]{Gruenb1}).\\

Now we introduce \textit{polytopal joins} and the \textit{quotient polytope}. Polytopal joins are needed for Lemmata from Reiner / Welker and Athanasiadis which we want to use and give furthermore another nice description of direct sums of polytopes. The quotient polytope is needed since basis polytopes are quotient polytopes and it allows us to present a lemma of Reiner and Welker which is essential to prove that two particular conditions are equivalent to the existence of a special simplex in a given polytope (Proposition \ref{SatzAequivSpezSimp}).\\

For polytopal and simplicial joins cf. \cite[p. 167]{Atha1} and \cite[p. 259, Definition 3.7]{ReiWel1} (observe that in all references only simplicial joins are used.):
\begin{defn}
Let $\tri_1, \tri_2$ be two polytopal complexes with disjoint vertex sets and affine basis spaces with trivial intersection in the interior. Then we define the \emph{polytopal join} $\tri_1 * \tri_2$ of $\tri_1, \tri_2$ as the polytopal complex containing the following facets $F$:
\begin{eqnarray*}
	F	& =	& \conv\{F_1,F_2\} \textrm{ with } \forall i \in \{1,2\}: F_i \textrm{ is a maximal face of } \tri_i.	
\end{eqnarray*}
The former faces of $\tri_1 * \tri_2$ are induced by the intersections of the facets (cf. \cite[p. 1193]{Tza1}).

If $\tri_1, \tri_2$ are simplicial complexes, we will call $\tri_1 * \tri_2$ a \emph{simplicial join}.
\label{DefSimplVerein}
\end{defn}

Note that for polytopal complexes $\tri_1, \tri_2$ we have $\dim(\tri_1 * \tri_2) = \dim(\tri_1) + \dim(\tri_2) + 1$.
If one of the two polytopal complexes $\tri_1, \tri_2$ (w.l.o.g.: $\tri_1$) of a polytopal join contains only the empty set, we have $\emptyset * \tri_2  = \tri_2$. This statement is in accord with the one about the dimension since $\dim(\emptyset) = -1$.

Let $\tri_1,\tri_2$ and $\tri_3$ be polytopal complexes whose affine basis spaces share no common non trivial subspace pairwise. Then the polytopal join is associative, i.e.:
\begin{eqnarray*}
	(\tri_1 * \tri_2) * \tri_3 	& =	& \tri_1 * (\tri_2 * \tri_3).
\end{eqnarray*}

Observe that we have for two polytopes $P \in \fR^n,Q \in \fR^m$ that
\begin{eqnarray*}
	\cF(\partial (P \oplus Q))	& =	& \cF(\partial P) * \cF(\partial Q)
\end{eqnarray*}
This is easy to see: Let $F_1 \in \cF(P), F_2 \in \cF(Q)$ facets of $P$ and $Q$ and let $A_P,A_Q$ denote the affine basis spaces of $P$ resp. $Q$. Then $F_1 \cap F_2 = \emptyset$ and $F_1 \not\parallel F_2$ since $A_P$ and $A_Q$ do online intersect in one point which is located in the interior of $P$ and $Q$. Hence $\conv(\cV(F_1) \cup \cV(F_2))$ is contained in a hyperplane $H$ of $\fR^{n+m}$. Since $H$ supports $F_1$ and $F_2$ all vertices $\cV(P \bs F_1)$ and $\cV(Q \bs F_2)$ are contained in $H^+$. Thus, $H$ is a supporting hyperplane of $P \oplus Q$. The inverse way follows directly from the definition of the direct sum.

\begin{defn}
Let $P \subseteq \fR^q$ be an $n$--polytope and $V$ an arbitrary subspace of $\fR^q$. Then the \emph{quotient polytope} $P / V \subseteq \fR^q / V$ is defined by
\begin{eqnarray*}
	P / V	& :=	& \{p + V \ | \ p \in P\}.
\end{eqnarray*}
Obviously $P / V$ can be identified with the image $\pi(P)$ of $P$ under a linear embedding $\pi : \fR^q \ra \fR^{q - \dim V}$ with kernel $V$ (cf. \cite[p. 167]{Atha1}, \cite[p. 260, Definition 3.10]{ReiWel1}).
\label{DefQuotPolytop}
\end{defn}
Notice that the basis polytope from Definition \ref{DefBasispolytop} hence is a particular quotient polytope.\\

With this definition we are able to formulate the following essential lemma (cf. \cite[ibid.]{Atha1}, cf. also \cite[p. 261, Proposition 3.12]{ReiWel1}):
\begin{lem}[Reiner, Welker / Athanasiadis]
Let $P$ be a $n$--polytope in $\fR^q$ with a triangulation $\tri$ isomorphic to $\Sig * \tri'$ with $\tri'$ denoting a (not exactly specified) simplicial complex and $\Sig$ a simplex which is not completely contained in the boundary of $P$. Let $V$ be the linear subspace of $\fR^q$ being parallel to the affine span of $\Sig$. Then the boundary complex of the quotient polytope $P / V \subseteq \fR^q / V$ is isomorphic to $\cF(P) \bs \cV(\Sig)$ and has a triangulation isomorphic to $\sig * \tri'$ with $\sig$ being an interior point of $P / V$.
\label{Lem2ReiWel1}
\end{lem}
The upper lemma was originally given by Reiner and Welker in \cite{ReiWel1}. Later it was cited and modified by  Athanasiadis in \cite{Atha1} into the upper version which we want to use here.

\subsection{Examples}

Before we start to prove our results we recall some of the known examples of polytopes with special simplices.

Be aware that special simplices of a polytope $P$ may never completely lie in the boundary, since in this case they would be contained in one particular facet of $P$ in contradiction to the Definition \ref{DefSpezSimpl}.

\begin{exa}
There are quite few combinatorially different polytopes with special simplices in low dimensions. Here we present two of them:
\begin{figure}[H]
	\begin{center}
	\begin{picture}(380,190)(0,0)
		\put(0,0){\includegraphics[scale=0.38]{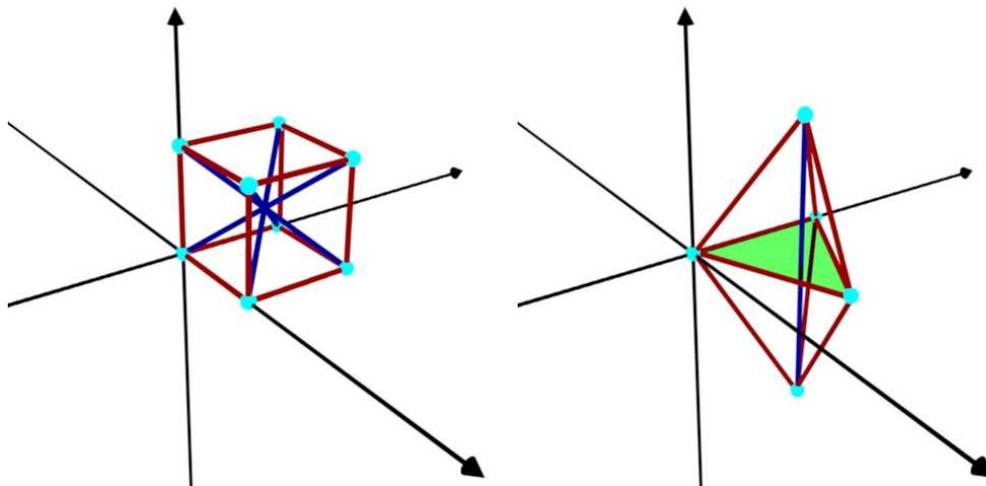} }
	\end{picture}
	\caption[The 3--cube and the bipyramid of a triangle with their special simplices.]{Left: the four special 1--simplices (blue) in a 3--dimensional cube; right: the special 1-- (blue) and 2--simplex (green) in a 3--dimensional bipyramid with triangular base (both lattice models).}
	\label{AbbSpezielleSimplizesLowDim}
	\end{center}
\end{figure}
\label{ExaSpezSimp1}
\end{exa}

The set of all doubly stochastic $n \times n$ matrices (i. e.: matrices with positive entries such that all row-- and column sums are 1) is a polytope in $\fR^{n \times n}$ --- the \emph{Birkhoff polytope} $P_n$.

Birkhoff polytopes are $0 / 1$--polytopes; they have $n!$ vertices (the permutation matrices of $S_n$), $n^2$ facets (in each case consisting of the set of all matrices with $x_{ij} = 0$) and dimension $(n-1)^2$. A complete classification is given by (cf. \cite[p. 20]{Zieg1}):
\begin{eqnarray*}
	P_n	& =	& \lf\{X \in \fR^{n \times n} \ \middle| \ \forall i,j \in \{1,\ldots,n\}: x_{ij} \geq 0 \w \sum_{k = 1}^n x_{ik} = \sum_{k = 1}^n x_{kj} = 1 \ri\}.
\end{eqnarray*}

\begin{exa}
Let $P_n$ be the $n$--th Birkhoff polytope. Let $v_1,\ldots,v_n$ be the $n \times n$ permutation matrices, which correspond to the elements of the cyclic subgroup of the symmetric group generated by the permutation $(1 \ 2 \ \cdots \ n)$  (or $n$ arbitrary permutation matrices with pairwise disjoint support). Then $v_1,\ldots,v_n$ are the vertices of a special simplex in $P_n$, because every facet of $P_n$ is defined by an equation of the form $x_{ij} = 0$ in $\fR^{n \times n}$ and does not contain exactly one of the vertices $v_1,\ldots,v_n$. This is obvious since the $v_k$ have the form:
$$\lf(\begin{array}{cccccc}
 0 	& 1 	 & 0 	  & \cdots	& 0	 & 0 	  \\
 0 	& 0 	 & 1	  & \ddots	& 0	 & 0 	  \\
 \vdots & \vdots & \ddots & \ddots	& \ddots & \vdots \\
 0	& 0	 & 0	  & \ddots	& 1	 & 0 	  \\
 0	& 0	 & 0 	  & \cdots	& 0	 & 1	  \\
 1	& 0	 & 0 	  & \cdots	& 0	 & 0 	  \\
\end{array}\ri),
\lf(\begin{array}{cccccc}
 0 	& 0 	 & 1 	  & \cdots	& 0	 & 0 	  \\
 \vdots & \vdots & \ddots & \ddots	& \ddots & \vdots \\
 0 	& 0 	 & 0	  & \ddots	& 1	 & 0 	  \\
 0	& 0	 & 0	  & \cdots	& 0	 & 1 	  \\
 1	& 0	 & 0 	  & \cdots	& 0	 & 0	  \\
 0	& 1	 & 0 	  & \cdots	& 0	 & 0 	  \\
\end{array}\ri),
\ldots,$$
i.e. in every row the 1 runs through every column successively (cf. \cite[p. 165]{Atha1}).
\label{Exa1Atha1}
\end{exa}

Let $\Omega := \{x_1,\ldots,x_n\}$ be a poset with respect to the relation $\leq_{\Omega}$. Then the set of all maps $f : \Omega \ra \fR$ forms an $n$--dimensional $\fR$--vectorspace (we identify every $f$ with its image vector). Using this fact the \emph{order polytope} $\cO(\Omega)$ is defined as the set of all $f(\Omega) \subset \fR^n$ which satisfy (cf. e.g. \cite{Stan2}):
\begin{enumerate}
	\item[(a)] $\forall x \in \Omega : 0 \leq f(x) \leq 1,$
	\item[(b)] $\forall x,y \in \Omega : \lf(\neg \exists z: z \neq x \w z \neq y \w x \leq_{\Omega} z \leq_{\Omega} y\ri) \Ra f(x) \leq f(y).$
 \end{enumerate}
The order polytope $\cO(\Omega)$ has dimension $n$. $f \in \cO(\Omega)$ is contained in a specific facet of the order polytope if and only if one of the following condition holds:
\begin{enumerate}
 \item[(a)] $\exists x \in \Omega: f(x) = 0 \w x \textrm{ is minimal with resp. to } \leq_{\Omega}.$
 \item[(b)] $\exists x \in \Omega: f(x) = 1 \w x \textrm{ is maximal with resp. to } \leq_{\Omega}.$
 \item[(c)] $\exists x,y \in \Omega: x \neq y \w f(x) = f(y) \w y \textrm{ covers } x \textrm{ with resp. to } \leq_{\Omega}$
 \end{enumerate}
(recall that in posets $(\Omega,\leq_{\Omega})$ $y$ \emph{covers} $x$ if and only if there is no $z \in \Omega \bs \{x,y\}$ with $x \leq_{\Omega} z \leq_{\Omega} y$).

An \emph{(order) ideal} or \emph{filter} in $\Omega$ is a subset $I \subseteq \Omega$ satisfying the following condition:
\begin{eqnarray*}
	\forall i,j \in \Omega: \lf(i \leq_{\Omega} j \w i \in I\ri)	& \Ra	& j \in I.
\end{eqnarray*}
The vertices of $\cO(\Omega)$ are the images of the filters of $\Omega$ under the characteristic function (with the filters interpreted as vectors with respect to the relation $\leq_{\Omega^0}$) (cf. \cite[p. 10 et seq.]{Stan2}).

\begin{exa}
Let $\Omega:= \{x_1,\ldots,x_n\}$ be a graded poset with respect to the relation $\leq_{\Omega}$ with rank $m-1$ ($1 \leq m \leq n$; i.e.: the maximal chain consists of $m-1$ elements). Let $P := \cO(\Omega)$ be the order polytope of $\Omega$ in $\fR^n$. Let $v_i$ be the characteristic vector of the filter consisting of all elements in $\Omega$ with rank greater than $m - i$  for all $i \in \{1,\ldots,m\}$ ($v_i$ is therefore a vertex of $P$). The upper equations defining the facets of an order polytope yield that $v_1,\ldots,v_m$ are the vertices of a special simplex in $P$ (cf. \cite[p. 165]{Atha1}).

To substantiate the intuition of order polytopes let $\Omega := \{x_1,x_2,x_3\}$ and look at the three order polytopes $\cO_{(1,1,1)}, \cO_{(2,1)}$ and $\cO_{(1,2)}$ given by the following three order structures:
\begin{eqnarray*}
	\begin{xy}
   	\xymatrix{
			x_3		\\
			x_2 \ar@{-}[u]	\\
			x_1 \ar@{-}[u]
 		}
  	\end{xy}
	\quad
  	\begin{xy}
   		\xymatrix{
					& x_3	&		\\
			x_1 \ar@{-}[ru] &	& x_2 \ar@{-}[lu]
 		}
  	\end{xy}
	\quad
  	\begin{xy}
   		\xymatrix{
			x_2	&				& x_3	\\
				& x_1 \ar@{-}[lu] \ar@{-}[ru]	&
 		}
  	\end{xy}
\end{eqnarray*}
Then one can show easily that $\cO_{(1,1,1)}$ is a simplex and its special simplex is the polytope itself; $\cO_{(2,1)}$ and $\cO_{(1,2)}$ are both 3--dimensional polytopes in form of a pyramid with square base each containing a special 2--simplex:
\begin{figure}[H]
	\begin{center}
	\begin{picture}(400,200)(0,0)
		\put(0,0){\includegraphics[scale=0.4]{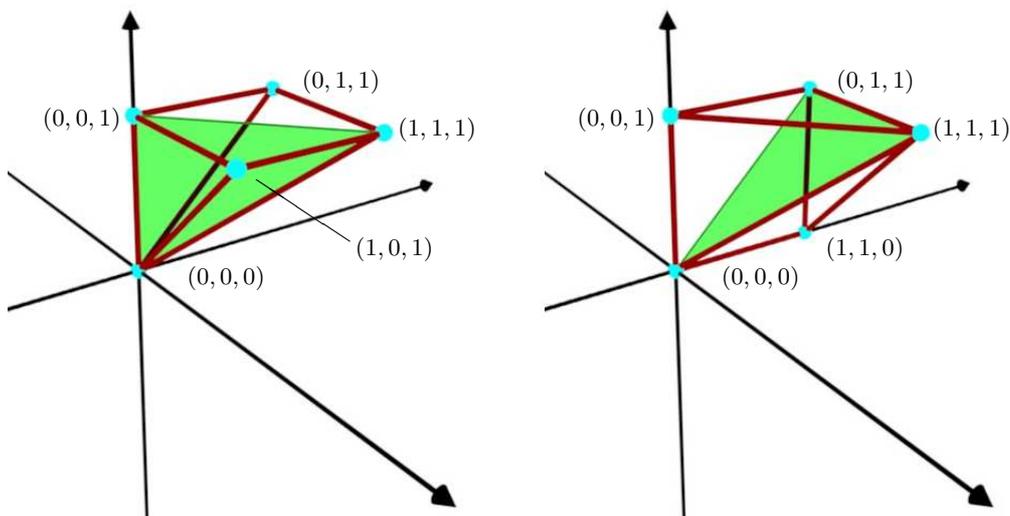}}
		\put(67,89){\mbox{\scriptsize{$(0,0,0)$}}}
		\put(13,149){\mbox{\scriptsize{$(0,0,1)$}}}
		\put(110,163){\mbox{\scriptsize{$(0,1,1)$}}}
		\put(146,145){\mbox{\scriptsize{$(1,1,1)$}}}
		\put(130,100){\mbox{\scriptsize{$(1,0,1)$}}}
		\put(128,105){\line(-3,2){35}}
		\put(267,89){\mbox{\scriptsize{$(0,0,0)$}}}
		\put(213,149){\mbox{\scriptsize{$(0,0,1)$}}}
		\put(310,163){\mbox{\scriptsize{$(0,1,1)$}}}
		\put(346,145){\mbox{\scriptsize{$(1,1,1)$}}}
		\put(306,100){\mbox{\scriptsize{$(1,1,0)$}}}
	\end{picture}
  	\caption[Two order polytopes with their special simplices.]{The order polytopes $\cO_{(2,1)}$ and $\cO_{(1,2)}$ with their special simplices (green) (lattice models).}
  	\label{AbbOrdnungspolytopeSpezSimp}
	\end{center}
\end{figure}
\label{Exa2Atha1}
\end{exa}

Raman Sanyal advised me to the class of weak Hannar polytopes: A polytope $P$ is called \emph{weakly Hannar} if $P = \conv(\cV(F) \cup \cV(-F))$ for every facet $F$ of $P$. Obviously every weak Hannar polytope is centrally symmetric (cf. \cite{Hansen1}, \cite{SanWerZieg1}).
\begin{exa}
In a weak Hannar polytope every pair of points $(x,-x)$ forms a special 1--simplex. This is easy to see since for every facet $F$ $x$ is contained in $F$ or $-F$ by definition. If $x \in F$ we have $-x \in -F$ due to the fact that weak Hannar polytopes are centrally symmetric and for the same reason there is no facet with $x,-x \in F$. Hence $\{x,-x\}$ is a special simplex.
\label{ExaWeaklyHannar}
\end{exa}

\section{Equivalence conditions for polytopes with special simplices}
\label{AbsExaUndAequivSatz}

Although the original Definition \ref{DefSpezSimpl} of special simplices is easy to understand it is not very easy to work with. Hence in Proposition \ref{SatzAequivSpezSimp} we introduce two conditions which are together equivalent to the existence of a special simplex. This Proposition \ref{SatzAequivSpezSimp} is the foundation to establish our classification of meek polytopes with special simplex. To prove it we need the following lemma of Athanasiadis (cf. \cite[p. 167 et seq.]{Atha1}):

\begin{lem}[Athanasiadis]
Let $P$ be an $n$--dimensional polytope and let $\tau := (v_p,\ldots,v_1)$ be an ordering of the vertices of $P$ such that $\cV(\Sig) := \{v_1,\ldots,v_k\}$ is the vertex set of a special simplex $\Sig$ in $P$. Let $\tri$ be the simplicial complex on the set $\{v_{k+1},\ldots,v_p\}$ induced by the reverse lexicographic triangulation $\tri_{\tau'}(\cF(P) \bs \cV(\Sig))$ with respect to $\tau' := (v_{k+1},\ldots,v_p)$. Then the following assertions hold:
\begin{enumerate}
	\item[(a)] The reverse lexicographic triangulation $\tri_{\tau}(P)$ of $P$ with respect to $\tau$ is isomorphic to the simplicial join $\Sig * \tri$.
	\item[(b)] $\tri$ is isomorphic to the boundary complex of a simplicial polytope of dimension $n - k + 1$.
\label{Lem2Atha1}
\end{enumerate}
\end{lem}

We can now present our equivalence proposition. The two conditions we mentioned above are firstly that the subcomplex of the original polytope $P$, which is left after eliminating all faces containing vertices of the special simplex, is the boundary complex of a \textit{polytope} $Q$ (the basis polytope) and secondly that \textit{every} hyperplane containing the special simplex divides $Q$. We express this formally:

\begin{prop}
Let $P$ be an $n$--polytope with $\cV(P) := \{v_0,\ldots,v_p\}$. Then $P$ contains a special $m$--simplex $\Sig$ (w.l.o.g.: $\cV(\Sig) = \{v_0,\ldots,v_m\}$ with $1 \leq m \leq n-1$) if and only if the following two conditions hold:
\begin{enumerate}
 \item[(a)] For every hyperplane $H \subset \fR^n$ with $\Sig \subset H$ holds:
 \begin{eqnarray*}
	\exists i,j \in \{m+1,\ldots,p\}: v_i \in H^+	& \w	& v_j \in H^-.
 \end{eqnarray*}
 \item[(b)] There is a basis polytope $Q \subset \fR^{n-m}$ (recall Definition \ref{DefBasispolytop}) such that
 \begin{eqnarray*}
	\cF(P) \bs \cV(\Sig) & =	& \cF(\partial Q).
 \end{eqnarray*}
\end{enumerate}
\label{SatzAequivSpezSimp}
\end{prop}

\begin{proof}
Assume w.l.o.g. that $v_0$ is the origin and therefore each of the observed $\Sig$ containing affine hyperplanes is a hyperplane of $\fR^n$.\\

\noindent Let $\Sig$ be a special simplex in $P$. Assume there is a hyperplane $H$ containing $\Sig$ with $v_i \in H^+$ for all $i \in \{m+1,\ldots,p\}$. Then we have in particular $P \subset H^+$. Since $\Sig \subset P$ and $\Sig \subset H$ this implies that $H$ is a bounding hyperplane of $P$ and $\Sig \subset \partial P$. But then there exists a facet containing $\Sig$ completely in contradiction to Definition \ref{DefSpezSimpl} of special simplices. This implies condition (a).

We additionally know by Lemma \ref{Lem2Atha1} that the reverse lexicographic triangulation $\tri_{\tau}(P)$ with respect to $\tau := (v_p,\ldots,v_0)$ of $P$ is isomorphic to the simplicial join $\Sig * \tri'$ (with $\tri' := \tri_{\tau'}(\cF(P) \bs \cV(\Sig))$). Therefore the conditions of Lemma \ref{Lem2ReiWel1} are fullfilled which in particular implies that $\cF(P) \bs \cV(\Sig)$ is isomorphic to the boundary of the quotient polytope $P / V \subset \fR^{n-m}$ with $V := \spn \{v_0,\ldots,v_m\}$. This yields condition (b).\\

Let us now assume that $P$ is a polytope satisfying conditions (a) and (b). Since boundary complexes of polytopes are pure we know by (b) that every maximal face $S$ of $\cF(P) \bs \cV(\Sig)$ has dimension $n - m - 1$. Every facet $F$ of $P$ with $F \bs \cV(\Sig) = S$ has dimension $n - 1$ by definition of $P$. Since eliminating a vertex from a facet (resp. more general: a face) lowers the dimension at most by 1, every facet of $P$ needs to contain $m$ affine independent vertices out of the set $\cV(\Sig)$.\\
But, condition (a) yields that no facet may contain all $m+1$ vertices from $\cV(\Sig)$, since otherwise for the appropriate bounding hyperplane $H$ all vertices of $\cV(P) \bs \cV(\Sig)$ would be located in $H^+$. Therefore, the vertices of $\Sig$ are affine independent, i.e. $\Sig$ is a simplex and every facet of $P$ contains exactly $m$ vertices of $\Sig$. This is exactly the definition of special simplex.
\end{proof}

If $P$ is a simplex itself it makes sense to say that $P$ has a special simplex (which is $P$ itself) as well by Definition \ref{DefSpezSimpl} as by the upper proposition since condition (a) and (b) are satisfied trivially. For the rest of this article we will not mention this special case anymore (the reader might just keep it in mind).

\begin{rem}
Note the following subtle fact in Proposition \ref{SatzAequivSpezSimp}: Although the polytopal complex $\cF(P) \bs \cV(\Sig)$ falls together with the boundary of an $(n-m)$--polytope $Q$, this does \textbf{not} in general imply that the affine basis space of this complex has the same dimension as $Q$. The dimension of the complex might be higher; the upper conclusion just applies to the combinatorial structure of  $\cF(P) \bs \cV(\Sig)$ and $\cF(\partial Q)$.
\label{RemPolymayproofild}
\end{rem}

We give an example to make this abstract remark clearer.
\begin{exa}
By Example \ref{ExaSpezSimp1}, Fig. \ref{AbbSpezielleSimplizesLowDim} a 3--cube contains a special 1--simplex. By eliminating the two vertices belonging to the simplex, one can see that the remaining complex is not contained in $\fR^2$, although it is the boundary complex of a 6--gon.
\label{ExaWuerfelistwild}
\end{exa}

This observation motivates the distinction between \textit{meek} and \textit{wild} polytopes with special simplices introduced in Definition \ref{DefZahmWild} in the introduction.\\

Clearly a polytope $P$ with special simplex $\Sig$ and basis polytope $Q$ is meek if and only if the intersection of the affine basis spaces of $\Sig$ and $Q$ is a single point. Furthermore, all polytopes $P$ with special simplex whose basis polytope $Q$ is a simplex are meek.

\section{Pyramids and direct sums}
\label{AbsPyrundMuscheln}
Before we use Proposition \ref{SatzAequivSpezSimp} to show how to construct polytopes out of $\PM_{(n,m)}$ (with $n,m$ arbitrary) we first deduce two corollaries to show that $\PM_{(*,1)}$ consists of bipyramids and to classify the low dimensional cases $\PM_{(2,1)}, \PM_{(3,2)}$ and $\PM_{(3,1)}$. Afterwards we show that constructing the pyramid of a meek $n$--polytope with special $m$--simplex yields a meek $(n+1)$--polytope with special $(m+1)$--simplex. Unfortunately one can not construct all polytopes in $\PM_{(n,m)}$ this way. Therefore we have to look on direct sums additionally. Both together deliver the whole class $\PM_{(n,m)}$ which we show in the next section. Notice that all polytopes with special simplex we look at until Section \ref{SubSecWildePolytope} are \textbf{meek}.

\begin{cor}
Let $P$ be a meek $n$--polytope with special simplex $\Sig$. 
If $\dim(\Sig) = 1$ then $P$ is a bipyramid, i.e. there is a $(n-1)$--basis polytope $Q$ with $f$--vector $(1,f_{q_0},\ldots,f_{q_{n-2}},1)$ such that we have for the $f$--vector $f_P := (1,f_{p_0},\ldots,f_{p_{n-1}},1)$ of $P$:
\begin{eqnarray*}
	f_{p_0} \, = \, f_{q_0} + 2, \ f_{p_{n-1}} \, = \, 2 \cdot f_{q_{n-2}}, \ \forall i \in \{1,\ldots,n-2\}: f_{p_i} \, = \, f_{q_i} + 2 \cdot f_{q_{i-1}}.	
\end{eqnarray*}
In other words:
\begin{eqnarray*}
	\PM_{(n,1)}	& =	& \{\bpyr Q \ | \ Q \text{ is basis polytope, } \dim(Q) = n-1\} \quad \text{ for all } n \in \fN.
\end{eqnarray*}
\label{KorSpezSimp1}
\end{cor}

\begin{proof}
Let $Q$ with $\dim(Q) = n-1$ be chosen arbitrary, $\cV(\Sig) = \{v_0,v_1\}$. Let $H_Q$ be the hyperplane containing $Q$. To achieve condition (a) in Proposition \ref{SatzAequivSpezSimp} for every hyperplane containing $\cV(\Sig)$, one vertex has to be in $H_Q^+$ and one in $H_Q^-$ and the vertices may not be connected (besides $Q$) by an edge. This corresponds exactly to a bipyramid with basis polytope $Q$ since $\cF(P) \bs \cV(\Sig)$ lies in $H_Q$ because $P$ is meek.
\end{proof}

We may now classify all meek polytopes with special simplices in low dimensions.
\begin{cor}
In $\fR^2$ exactly all convex quadrangles have a special 1--simplex. In $\fR^3$ exactly the polytopes with $f$--vector $(1,5,9,6,1)$ or $(1,5,8,5,1)$ have a special 2--simplex and all arbitrary bipyramids are meek and have a special 1--simplex, i.e.:
\begin{eqnarray*}
  \PM_{(2,1)} & = & \{P \ | \ f_P = (1,4,4,1)\}, \\
  \PM_{(3,2)} & = & \{P \ | \ f_P = (1,5,9,6,1) \vee f_P = (1,5,8,5,1)\}, \\
  \PM_{(3,1)} & = & \{P \ | \ f_P = (1,k+2,3 \cdot k,2 \cdot k,1), k \in \fN_{\geq 3}\}.
\end{eqnarray*}
\label{KorSpezSimp2}
\end{cor}

\begin{proof}
The assertions on $\PM_{(2,1)}$ and $\PM_{(3,1)}$ may be deduced from the previous corollary by a simple calculation.

For $\PM_{(3,2)}$ the conditions in Proposition \ref{SatzAequivSpezSimp} yield that there is exactly one hyperplane containing $\Sig$ and $\cF(\partial Q) = \cF(P) \bs \cV(\Sig)$ consists of merely two vertices since $\dim(Q) = 1$. Thus, the intersection of $Q$ and $\Sig$ is either in the interior of $\Sig$ or on a 1--face of $\Sig$. These are the two cases mentioned above. The intersection may not be a vertex since this vertex would no longer be part of $\cV(P)$ and the intersection may not be empty since $\Sig$ would not be a special simplex otherwise.
\end{proof}

Note the following further fact Proposition \ref{SatzAequivSpezSimp} yields: For all $n \in \fN, m \in \fN_{<n}$ there exists a canonical embedding $\rho: \PM_{(n,m)} \ra \PM_{(n+1,m+1)}$ which is induced by a pyramid construction. This leads to the first way to construct a meek polytope with special simplex.
\begin{cor}
 Let be $P$ a meek $n$--polytope ($\cV(P) := \{v_0,\ldots,v_p\}$) with special $m$--simplex. Then its pyramid $\pyr P \subset \fR^{n+1}$ has a special $m+1$--simplex and is meek.
\label{KorSpezSimp3}
\end{cor}

\begin{proof}
Let $\cV(\Sig) := \{v_0,\ldots,v_n\}$ and let $w$ be the vertex which is added as the apex of the pyramid of $P$, i.e.:
\begin{eqnarray*}
	\text{pyr } P	& =	& \lf(\conv \cV(P) \cup \{w\}\ri).
\end{eqnarray*}
Then $\Sig' := \conv \lf(\cV(\Sig) \cup \{w\}\ri)$ is a special simplex in pyr $P$: since if $H'$ is a hyperplane in $\fR^{n+1}$ with $\Sig' \subset H'$ then there exists one and only one hyperplane $H$ of $\fR^n$ with $H = H' \cap \fR^n$ and $\Sig \subset H$. Since $\Sig$ is a special simplex of $P$, condition (a) of Proposition \ref{SatzAequivSpezSimp} yields:
\begin{eqnarray*}
	\exists i,j \in \{m+1,\ldots,p\}: v_i \in H^+ \w v_j \in H^-
\end{eqnarray*}
and thus, due to the fact that $\{v_{m+1},\ldots,v_p\} \subset \fR^n$, by construction of $H'$:
\begin{eqnarray*}
	\exists i,j \in \{m+1,\ldots,p\}: v_i \in H'^+ \w v_j \in H'^-.
\end{eqnarray*}
Therefore condition (a) of Proposition \ref{SatzAequivSpezSimp} holds for pyr $P$ and $\Sig'$. Condition (b) is trivial. Hence $\Sig'$ is special $m+1$--simplex of pyr $P$.
\end{proof}

The two corollaries \ref{KorSpezSimp1} and \ref{KorSpezSimp3} put together provide the simple first method to construct an $n$--dimensional polytope with special $m$--simplex: 

\begin{cor}
Let $Q$ be an arbitrary polytope with dimension $n - m$. Then
\begin{eqnarray*}
	\py{m-1} (\bpyr Q)	& :=	& \underbrace{\pyr \pyr \cdots \pyr }_{m-1 \text{ times}} \bpyr Q
\end{eqnarray*}
is a meek $n$--polytope with special $m$--simplex.
\label{KorSpezSimp4}
\end{cor}

We will always assume from now on w.l.o.g. that the apex of the pyramid $\pyr P$ has the coordinates $(0,\ldots,0,1)^T$.\\

Unfortunately, in the class of meek polytopes with special simplices one cannot construct $\PM$ in total just by constructing pyramids of bipyramids of arbitrary polytopes. The reason is that in an $n$--pyramid with special $m$--simplex $\Sig$ there has to be an $(m-1)$--face $F$ of $\Sig$ which is contained in the base of the pyramid but not in the boundary of the base of the pyramid. Thus, $F$ may not be an element of the face lattice of the pyramid (note that in low dimensions this may only appear with special 2--simplices since vertices need to be in the face lattice). But there are examples of polytopes with special simplex $\Sig$ which have every element of the face lattice of $\Sig$ in $\cF(P)$: for example the bipyramid of a 2--simplex (this is the case $(1,5,9,6,1)$ in Corollary \ref{KorSpezSimp2} which was already introduced in Example \ref{ExaSpezSimp1}; the 2--simplex is the interesting one here) or the Birkhoff polytope $P_3$ of all doubly stochastic $3 \times 3$ matrices --- a 4--dimensional polytope with special 2--simplex whose face lattice is completely contained in the face complex of $P_3$.

Hence there has to be another method to create polytopes with special simplex which avoids the problem mentioned above. It turns out that this method is the direct sum:

\begin{prop}
The direct sum $\Sig \oplus Q$ of an $k$--simplex $\Sig$ and an $m$--polytope $Q$ is a meek $(m+k)$--polytope and always contains $\Sig$ as a special $k$--simplex. Furthermore
\begin{eqnarray*}
	\forall F \in \cF(\partial \Sig) : F \in \cF(\Sig \oplus Q),
\end{eqnarray*}
holds, i.e. all elements of the boundary complex of $\Sig$ are contained in the face complex of the direct sum $\Sig \oplus Q$.
\label{SadefShell}
\end{prop}

\begin{proof}
Recall that $\cF(\partial (\Sig \oplus Q)) = \cF(\partial \Sig) * \cF(\partial Q)$. By definition of the polytopal join every face of $\cF(\partial \Sig) * \cF(\partial Q)$ is the convex hull of a facet $F_1$ of $\cF(\partial Q)$ and a facet $F_2$ of $\cF(\partial \Sig)$. Since the affine basis spaces of $Q$ and $\Sig$ intersect in only one point located in the interior of both complexes all vertices of both facets have to be contained in $\conv(\cV(F_1) \cup \cV(F_2))$. Thus, $\Sig \oplus Q$ contains a special simplex since $\Sig$ is simplex and hence every facet of $\Sig$ contains all but one vertex of $\Sig$. 

By definition of the direct sum it follows that $\cF(P) \bs \cV(\Sig)$ is contained in the affine basis space of $Q$ and therefore that $\Sig \oplus Q$ is meek. 

By definition of the polytopal join the convex hull of every pair of a facet of $Q$ and a facet of $\Sig$ is a facet of $\Sig \oplus Q$. This implies in particular that every facet and hence every face of $\Sig$ is contained in the boundary of $\Sig \oplus Q$.
\end{proof}

Observe that every pyramid of a polytope $P$ with special $m$--simplex $\Sig_m$ and basis polytope $Q$ may easily be transformed into $\Sig_m \oplus Q$ geometricly. One just has to push $\Sig_m$ out of the pyramids base into the new dimension emerged during pyramid construction in the opposite direction of the apex of the pyramid. 

Observe, on the other way around, that one may transform $\Sig_k \oplus Q$ into $\py{j}(\Sig_{k-j} \oplus Q)$ by intersecting $\Sig_k \oplus Q$ with $j$ hyperplanes such that every hyperplane is parallel to a facet $F$ of $\Sig_k$, contains $Q$ and is oriented such that $F$ is in the negative halfspace. Notice that a hyperplane is determined that way since, if $\dim(\Sig_k \oplus Q) = n$, then it has to contain the intersection point of $Q$ and $\Sig_k$, the basis space of $Q$ has dimension $n - k$ and the facet, it has to be parallel to, has dimension $k - 1$.\\

The next step is to explain the combinatorical structure (i.e. here: the $f$--vector) of $\Sig_k \oplus Q$ for an arbitrary polytope $Q$. But this is easy due to $\cF(\partial (\Sig_k \oplus Q)) = \cF(\partial Q) * \cF(\partial \Sig_k)$. This means pictorially that (for every $j \in 1 \leq i + \dim(Q)$) the $j$--faces of the face lattice of $\Sig_k \oplus Q$ are all the combinatorially combinable ones from the lattices $\cF(\partial \Sig_k)$ and $\cF(\partial Q)$. This is given more specificly by the following proposition:

\begin{prop}
Let $Q$ be a $q$--polytope with $f$--vector $f_Q :=	(1,f_{Q_0},f_{Q_1},\ldots,f_{Q_{q-1}},1)$. Then the $f$--vector $f_{\Sig_k \oplus Q}$ of $\Sig_k \oplus Q$ is given by
\begin{eqnarray*}
	f_{(\Sig_k \oplus Q)_j}	& =	& \sum_{r,l : r + l = j - 1} f_{Q_r} \cdot \binom{i+1}{l+1} \quad \text{ for all } j \in \{0,\ldots,q+i\}.
\end{eqnarray*}
\label{SatzKombStrukShell}
\end{prop}

\begin{proof}
By the definition of direct sums we know that the intersection of the affine basis spaces of $Q$ and $\Sig_k$ is a single point which is not part of the boundary of $\Sig_k$. Furthermore, we know that all elements of $\cF(\partial \Sig_k)$ are in $\cF(\Sig_k \oplus Q)$. Since $\Sig_k$ is a special simplex, condition (b) of Proposition \ref{SatzAequivSpezSimp} implies that all elements of $\cF(\partial Q)$ are in $\cF(\Sig_k \oplus Q)$ either. Thus, every pair of an $r$--face of $\cF(\partial \Sig_k)$ and an $l$--face of $\cF(\partial Q)$ forms a $r + l + 1$--face of $\cF(\Sig_k \oplus Q)$. Since the number of $l$--faces in an $k$--simplex is $\binom{k+1}{l+1}$ the assertion follows.
\end{proof}

\section{Full classification of meek polytopes with special simplices}
\label{AbsThmKlassSpezSimp}
After discussing the relation between direct sums and pyramids (of bipyramids) of polytopes we finally need to check if there exist additional meek polytopes with special simplices which are not obtained by these two constructions. The following main theorem (which was already given roughly as Theorem \ref{ThmKlassSpezSimpRough} in the introduction) yields that this is not the case and therefore gives a full classification of $\PM_{(n,m)}$ since the $f$--vectors of pyramids and direct sums are known:

\begin{thm}
Let $P$ be an $n$--polytope. Then $P$ is meek and contains a special $m$--simplex $\Sig$ if and only if $P$ has an $(n - m)$--basis polytope $Q$ (cf. Definition \ref{DefBasispolytop}) such that
\begin{eqnarray*}&&
	\begin{array}{cccccccc}
		\exists \, i \in \fN, j \in \fN_{>0} :	& i + j	& =	& m	& \text{ and }	& P 	& =	& \py{i}\lf(\Sig_j \oplus Q\ri);
	\end{array}
\end{eqnarray*}
with $\Sig_j$ denoting a $j$--simplex. In other words:
\begin{eqnarray*}
&& \forall n, m \in \fN: \PM_{(n,m)} =\\
&& \qquad \{P \ | \ \exists \, i, j \in \fN, \exists \, Q : i + j = m \w \dim(Q) = n - m \w P = \py{i}\lf(\Sig_j \oplus Q\ri)\}.
\end{eqnarray*}
\label{ThmKlassSpezSimp}
\end{thm}

\begin{proof}
Let $P$ be an arbitrary $n$--polytope with special $m$--simplex $\Sig$. According to condition (b) in Proposition \ref{SatzAequivSpezSimp} we know that $P$ has an $(n-m)$--basis polytope $Q$ and all elements in $\cF(Q)$ are elements in $\cF(P)$. $\Sig \cap Q$ is a single point which is contained in the interior of $Q$ since $P$ is meek and hence each facet of $P$ is the convex hull of a facet of $Q$ and a facet of $\Sig$ lying in the boundary complex of $P$. We know further that we are in the situation 
\begin{eqnarray*}
	P	& =	& \Sig \oplus Q
\end{eqnarray*}
if and only if all elements of $\cF(\partial \Sig)$ are elements of $\cF(P)$. If an $i$--face $F$ (with $\cV(F) := \{v_{F_1},\ldots,v_{F_r}\}$) of $\Sig$ is not an element of $\cF(P)$, then we know due to the construction of a pyramid (and the resulting affine dependence of $\cV(F)$ and $\cV(Q)$) that all $(i+1)$--, \ldots, $(m-1)$--faces (thus $i$--, $(i+1)$--,\ldots, $(m-1)$--simplices) containing $\cV(F)$ are not elements of $\cF(P)$ either. Hence we can realize this situation by
\begin{eqnarray}
	P	& =	& \py{i}\lf(\Sig_{m-i} \oplus Q)\ri) \label{EquMainThm1}.
\end{eqnarray}
Now we show that no further cases are possible: Let $A_Q$ and $A_\Sig$ denote the affine basis spaces of $Q$ and $\Sig$. Moving the vertices of $\cV(Q)$ in $A_Q$ changes either nothing of $\cF(P)$ or it also changes $\cF(Q)$ which is not allowed. Moving the vertices of $\Sig$ in $A_\Sig$ combinatorically only effects the position of the point where $\Sig$ intersects $Q$ and thus which faces of $\Sig$ are not contained in $\cF(P)$. Assume there is a further $i'$--face $G$ of $\Sig$ (with $\cV(G) := \{v_{G_1},\ldots,v_{G_s}\}$) next to $F$ not being an element of $\cF(P)$ and there is no $i$ such that (\ref{EquMainThm1}) holds. But then the $i''$--face $F \cap G$ is also not contained in $\cF(P)$. If $i'' \leq 1$ this is a contradiction, since $\cV(\Sig)$ has to be contained in $\cF(P)$ and otherwise we are directly in the situation
\begin{eqnarray*}
	P	& =	& \py{i''}\lf(\Sig_{m-i''} \oplus Q)\ri)
\end{eqnarray*}
in contradiction to our assumption since in this case all $(i''+ 1)$--, \ldots, $(m-1)$--faces of $\Sig$ containing $\cV(G) \cap \cV(F)$ may not be element of $\cF(P)$ again. 
\end{proof}

\section{Wild polytopes with special simplices}
\label{SubSecWildePolytope}
In Remark \ref{RemPolymayproofild} and Example \ref{ExaWuerfelistwild} we have demonstrated (by looking at the 3--cube) the necessity of the distinction between meek and wild polytopes with special simplex as we have defined it in Definition \ref{DefZahmWild}. After classifying the meek polytopes and demonstrating how to construct them geometricly in the former sections one would like to do the same for wild polytopes.

Unfortunately it is not possible to give an analogue complete classification of the wild polytopes with special simplex in general since the question how many exist for a specific basis polytope and how their $f$--vectors look like cannot be answered without additional information about the basis polytope itself.
 
What we can do, is to give a full classification in the sense that every polytope in $\cS_{(n,k)}(Q)$ can be constructed by intersecting $\Sig_k \oplus Q$ with certain hypersurfaces which we can determine. This Theorem \ref{ThmWildPolyCharac} will furthermore yield  upper bounds for the $f$--vector of wild polytopes with special simplex.\\

Firstly, we show that there is a bijection between the vertex sets of a polytope in $\cS_{(n,k)}(Q)$ and $\Sig_k \oplus Q$:
\begin{lem}
Let $P$ be an $n$--polytope with special $k$--simplex $\Sig$, basis polytope $Q$. Let $A_\Sig$ and $A_{\cF(P) \bs \cV(\Sig)}$ denote the affine basis spaces of $\cF(P) \bs \cV(\Sig)$ resp. $\Sig$ such that w.l.o.g. $0 \in A_{\cF(P) \bs \cV(\Sig)} \cap \Sig$. Let $\psi: A_\Sig \ra A_\Sig$ denote the linear map $v \mapsto w + (1 + \eps) \cdot (v-w)$ with $\eps > 0$ with $w$ denoting the center point of the normal fan of $\Sig$. Then $\psi$ and the projection $\pi: \fR^n \ra \fR^n / A_\Sig$ induce a bijection $\phi: \cV(P) \ra \cV(\Sig \oplus Q)$. We call $\phi$ the \emph{vertex projection} of $P$.
\label{LemVertexProjection}
\end{lem}

\begin{proof}
We have $\# \cV(P) = \# \cV(\Sig \oplus Q)$ since $\cF(P) \bs \cV(\Sig) = \cF(\partial Q)$ due to Proposition \ref{SatzAequivSpezSimp}. If $P$ is meek the lemma is trivial since in this case we have $A_{\cF(P) \bs \cV(\Sig)} \cap \Sig = \{0\}$ and $\pi_{|\cV(P)|} = \id$. $\psi$ ensures that $Q \cap \Sig$ intersects in the interior of $\Sig$ since it blows up $\Sig$ slightly. Thus, we obtain $\Sig_k \oplus Q$ due to Theorem \ref{ThmKlassSpezSimp} and we have $\phi(v) = v$ for $v \in \cV(P) \bs \cV(\Sig)$ and $\phi(v) = \psi(v)$ for $v \in \cV(\Sig)$.

If $P$ is wild Proposition \ref{SatzAequivSpezSimp} tells us that $\cF(P) \bs \cV(\Sig) = \cF(\partial Q)$ and Definition  \ref{DefBasispolytop} and Lemma \ref{Lem2ReiWel1} yield $\pi(\cF(P) \bs \cV(\Sig)) = \cF(\partial Q)$. Since we know that due to Theorem \ref{ThmKlassSpezSimp} there is one unique meek polytope with special simplex determined by $n$,$k$, $Q$, and the position of $\Sig \cap A_{\cF(P) \bs \cV(\Sig)}$ in $\Sig$, $\pi_{|\cV(P) \bs \cV(\Sig)}$ induces a bijection between $\cV(P)$ and the vertex set of a meek polytope in $\cM_{(n,k)}(Q)$. From this point one goes on as above.
\end{proof}

Now we are attempting to prove our main theorem which will be a generalization of Theorem \ref{ThmKlassSpezSimp}. Before we can do that we need to introduce two definitions:

\begin{defn}
Let $L(P)$ be the face lattice of an $n$--polytope $P$ in $\fR^n$ and $Q$ an $k$--polytope such that $\cF(\partial Q) \subseteq \cF(\partial P)$ (and thus $L(Q) \subseteq L(P)$). Let $H$ be a hyperlane with $H \cap Q \neq \emptyset$ which is combinatorially defined by the vertices of $P$.
\begin{enumerate}
	\item[(a)] We call $H$ $Q$--\emph{separating} if it is possible to deform $P$ into an $n$--polytope $P'$ such that there is a bijection between $\cV(P)$ and $\cV(P')$, $\cF(\partial Q) \subseteq \cF(\partial P')$ and the interior of every face of $Q$ is completely contained in $H$ or in $H^+$ or in $H^-$. We call $P'$ an $H$--\emph{realization} of $P$ (resp. $L(P)$)
	\item [(b)] If there is an $H$--realization $P'$ with $H^- \cap Q = \emptyset$, we call $H$ $Q$--\emph{supportable}.
	\item [(c)] If we have $L(P) = L(P')$ we say that $H$ separates (resp. supports) $Q$ \emph{trivially}.
	\item [(d)] Let $\cH$ be a set of hyperplanes. We call $\cH$ \emph{simultaneously} $Q$--\emph{separating / --supportable} if there is one deformation $P'$ of $P$ such that every $H \in \cH$ is $Q$--separating / --supportable with $H$--realization $P'$.
\end{enumerate}
\label{DefQSeperable}
\end{defn}

Observe that $Q$--separatability / --supportability is a pure combinatorial property of $P$ which does not depend on the geometry of $P$. Observe furthermore that this definition in particular implies that for an $Q$--separating hyperplane it is possible to arrange $P$ in a way such that the intersection of $H$ and an arbitrary face $F$ of $Q$ is a face of $Q$ again. If the separation is non--trivial, $H$ has to intersect $Q$ along certain connected ridges of $Q$. Finally observe that the cases $\dim(Q) \in \{1,2\}$ are easy: For $\dim(Q) = 1$ every $H$ separates $Q$ trivially. For $\dim(Q) = 2$ exactly every hyperplane intersecting $Q$ in a pair of vertices is $Q$--separating since $Q$ is an $n$--gon.\\

\begin{exa}~
\begin{enumerate}
	\item [(a)] Let $P = Q$ such that $P$ is the 3--polytope from figure \ref{AbbExaSeperation}. The orange hyperplane in the left picture separates the polytope since every face is completely contained in the hyperplane or in its positive / negative halfspace. The lilac hyperplane in the right picture does not separate the polytope since there are three facets which are divided into the positive and the negative halfspace along the blue lines and it is impossible to move vertices in or beneath or beyond the hyperplane without changing $L(P)$. Both hyperplanes are not $Q$--supportable since this is obviously never possible if $\dim(P) = \dim(Q)$.
		
	\begin{figure}[H]
	\begin{center}
	\begin{picture}(400,200)(0,0)
		\put(0,0){\includegraphics[scale=0.4]{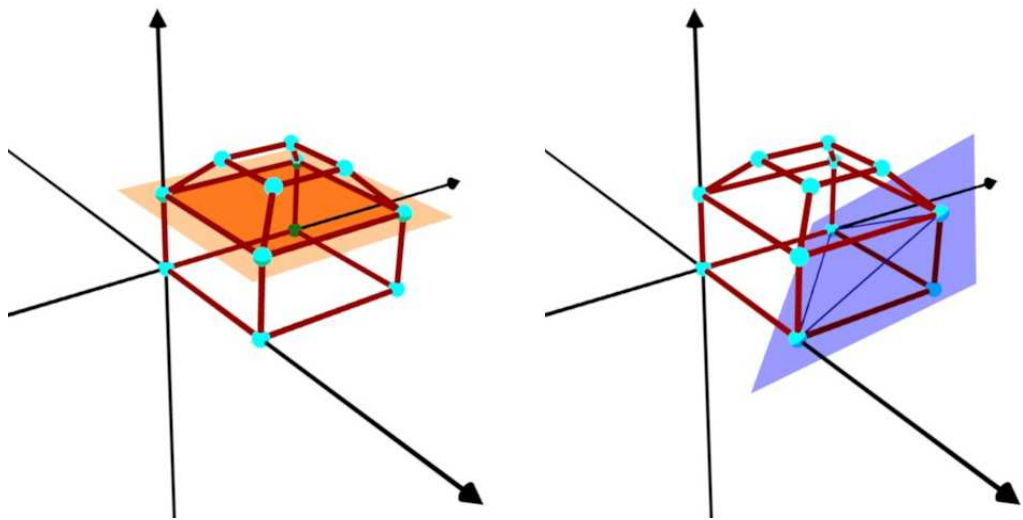}}
	\end{picture}
	\caption{A polytope with a separating and a non--separating hyperplane.}
	\label{AbbExaSeperation}
	\end{center}
	\end{figure}

	\item [(b)] Let $Q$ be the 3--zonotope given in the left picture of figure \ref{Abb4CubeBasispoly} and $P$ be the 4--polytope $\pyr Q$ with Schlegel--diagram given in the middle picture of figure \ref{Abb4CubeBasispoly} (both as lattice models; the green point is the apex of the pyramid; for additional information about Schlegel--diagrams see cf. \cite[p. 132 et. seq.]{Zieg1}). Then the set $\cH$ of four hyperplanes (3--dimensional) which are each given by the apex and three of the turquoise vertices of $Q$ which are intersecting in the six orange 2--planes in the right picture of figure \ref{Abb4CubeBasispoly} is simultaneously $Q$--separating as the right picture shows. It is even simultaneously $Q$--supportable as we will see at the end of this section.

	\begin{figure}[H]
	\begin{center}
	\begin{picture}(570,160)(0,0)
		\put(0,0){\includegraphics[scale=0.38]{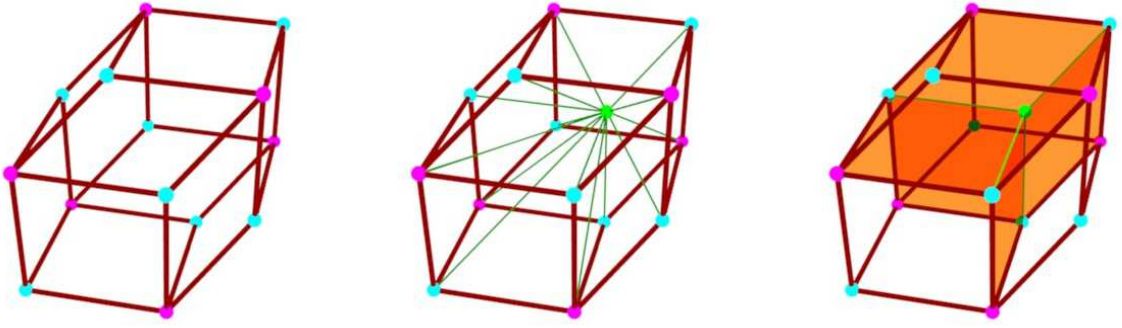}}
	\end{picture}
	\caption[A 3--polytope, its pyramid and a set of simultaneously supportable hyperplanes.]{A 3--polytope $Q$, its pyramid $P := \pyr Q$ and a set of simultaneously $Q$--supportable hyperplanes.}
	\label{Abb4CubeBasispoly}
	\end{center}
	\end{figure}
\end{enumerate}
\label{ExaQSeperable}
\end{exa}

\begin{defn}
Let $P$ be polytope with special simplex $\Sig$, basis polytope $Q$ and vertex projection $\phi$. Let $H$ be an arbitrary hyperplane. We define the $H$--\emph{corresponding} hyperplane $\wh H$  for $\Sig \oplus Q$ ($= \conv(\phi(\cV(P)))$) in the following way:
If we have $\cF(P) \bs \cV(\Sig) \subset H$ then we define $\wh H$ by $Q \subset \wh H$ and
\begin{eqnarray*}
	v \in \wh H		& \Lera	& v \in H  \text{ for } v \in \cV(\Sig).
\end{eqnarray*}
If $\cF(P) \bs \cV(\Sig) \not\subset H$ we define $\wh H$ as the affine hull of elements of $\cV(\Sig \oplus Q)$ such that by
\begin{eqnarray*}
	\phi(v) \in \wh H		& \Lera	& v \in H  \text{ for } v \in \cV(\Sig) \\
	\phi(w) \in \wh H \cap \wh H^-	& \Lera	& w \in H  \text{ for } w \in \cV(P) \bs \cV(\Sig).
\end{eqnarray*}
\label{DefCorrespHyperplane}
\end{defn}
That means the corresponding hyperplane intersects all those vertices $\phi(v)$ of $Q$ whose preimage $v$ is as well contained in a face of $Q$ in $H$ as in a face of $Q$ contained in $H^+$.\\

Observe that for every hyperplane $H$ there always exists a corresponding hyperplane since $\phi$ maps $\cF(P) \bs \cV(\Sig) \cap H$ to a connected part of $\cF(\partial Q)$ and all $v \in \cV(\Sig)$ are contained in the interior of the simplex spaned up by $\{\phi(v): v \in \cV(\Sig)\}$.

\begin{exa}
The green supporting hyperplane of the wild polytope on the left picture has a corresponding hyperplane in the direct sum of the special one simplex and the 8--gon which is given by the violet vertices. The blue part of the 8--gon is the part of $Q$ which is separated by the corresponding hyperperlane and thus is contained in its negative halfspace.
\begin{figure}[H]
	\begin{center}
	\begin{picture}(360,190)(0,0)
		\put(0,0){\includegraphics[scale=0.36]{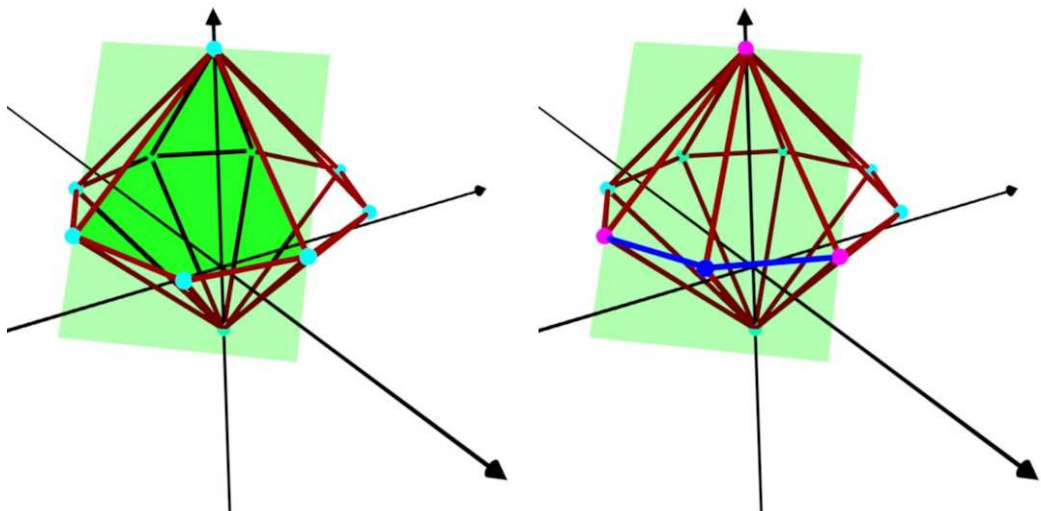}}
	\end{picture}
	\caption[A supporting hyperplane of a wild polytope and its corresponding hyperplane.]{A supporting hyperplane of a wild polytope and its corresponding hyperplane in the direct sum. The blue part of the 8--gon is separated by the corresponding hyperplane which intersects the direct sum in the pink vertices.}
	\label{AbbExaCorrespHyperplane}
	\end{center}
\end{figure}
\end{exa}

\begin{thm}
An $n$--polytope $P$ contains a special $k$--simplex $\Sig$ with basis polytope $Q$ (i.e.:  $P \in \PS_{(n,k)}(Q)$) if and only if the following conditions hold: Let $\cH$ denote the set of bounding hyperplanes of facets of $P$ and $\wh \cH$ the set of corresponding hyperplanes in $\Sig \oplus Q$:
\begin{enumerate}
	\item[(a)] $\wh \cH$ is simultaneously $Q$--supportable in $\Sig \oplus Q$.
 	\item[(b)] Every $\wh H \in \wh \cH$ is a bounding hyperplane of $\Sig \oplus Q$, contains $Q$ or there exist a facet $F$ of $\Sig \oplus Q$ such that $\ttint(F)$ is completely contained in $\wh H^-$.
	\item[(c)] 
	For every $\wh H \in \wh \cH$ there exists exactly one $w_0 \in \cV(\Sig) \cap \wh H^+$.
\end{enumerate}
Furthermore $P$ is meek if all corresponding hyperperlanes $\wh H$ separate $Q$ trivially.
\label{ThmWildPolyCharac}
\end{thm}

\begin{proof}
First we show that conditions (a) --- (c) are satisfied if $P$ is a polytope with special simplex. Due to vertex projection we have a bijection between $\cV(P)$ and $\cV(\Sig \oplus Q)$. Proposition \ref{SatzAequivSpezSimp} yields $\cF(P) \bs \cV(\Sig) = \cF(\partial Q)$. This implies that the corresponding hyperplanes of all bounding hyperplanes of $P$ have to be simultaneously $Q$--supportable since otherwise it would be impossible to move the vertices of $\wh H^-$ into $\wh H$ (simultaneously for every $\wh H \in \wh \cH$) without changing the combinatorical structure of $\cF(P) \bs \cV(\Sig)$; this is (a). 

Let $H \in \cH$. Let $v_1,\ldots,v_j \in \cV(P) \bs \cV(\Sig)$, $w_1,\ldots,w_k \in \cV(\Sig)$ denote the vertices of $P$ which are contained in $H$ and $M := \{\phi(v_1),\ldots,\phi(v_j)\} \cap \wh H^- \subset \cV(Q)$ denote the vertices of $Q$ which are contained in $\wh H^-$. Notice that $\{\phi(w_1),\ldots,\phi(w_k)\} \subset \wh H$ or $\{\phi(w_1),\ldots,\phi(w_k)\} \subset \wh H^-$. If $M = \emptyset$, i.e. if $\wh H$ separates $Q$ trivially, and $\cF(P) \bs \cV(\Sig) \not\subset H$ we have  $\{\phi(w_1),\ldots,\phi(w_k)\} \subset \wh H$ and $\wh H$ is a bounding hyperplane. If $M = \emptyset$ and $\cF(P) \bs \cV(\Sig) \subset H$, we have $\{\phi(w_1),\ldots,\phi(w_k)\} \subset \wh H^-$. Then $\phi(w_1),\ldots,\phi(w_k)$ forms a facet of $\Sig \oplus Q$ with every facet of $Q$. These are contained in $\wh H^-$. If $M \neq \emptyset$, then $\phi(v_1),\ldots,\phi(v_j)$ form at least two facets of $Q$ since $\wh H$ is $Q$--supportable and during the deformation from $\Sig \oplus Q$ to $P$ there are only vertices in $\wh H^-$ moved into $\wh H$ due to Definition \ref{DefCorrespHyperplane}. These facets form a facet of $\Sig \oplus Q$ together with $\phi(w_1),\ldots,\phi(w_k)$; this is (b).

Let $\cV(\Sig) = \{w_0,\ldots,w_k\}$. Since $\Sig$ is special simplex in $P$, we have w.l.o.g. $w_1,\ldots,w_k \in H$ and $w_0 \in H^+$. Since $\Sig \subset \conv(\{\phi(w_0),\ldots,\phi(w_k)\})$ we have $\phi(w_1),\ldots,\phi(w_k) \in \wh H \cup \wh H^-$ and $\phi(w_0) \in \wh H^+$ and (c) holds.\\

Now we move the vertices of a polytope $P$ with special $k$--simplex $\Sig$ and basis polytope $Q$ into a set of hyperplanes $\cH$ such that for every hyperplane $H \in \cH$ (a) --- (c) holds. We show that this yields a polytope $P'$ with special $k$--simplex and basis polytope $Q$ again. 

Since $\cH$ is simultaneously $Q$--supportable due to (a) all vertices may be moved into the hyperplanes $H \in \cH$ without changing the combinatorial structure of $\cF(P) \bs \cV(\Sig)$. Hence $P'$ has a basis polytope $Q$ again. 

Condition (b) yields that $H \cap P'$ is $(n-1)$--dimensional again since $H^-$ contains at least one facet of $P$, the vertices of $\Sig$ remain a simplex since $P'$ is an $n$--polytope and no vertex of $\Sig$ may be contained in the affine span of a face of $Q$ since otherwise the structure of $\cF(P) \bs \cV(\Sig)$ was changed.

For every $H \in \cH$ there is exactly one $w_0 \in H^+$ due to (c). Thus $\{w_1,\ldots,w_k\} := \cV(\Sig) \bs \{w_0\} \subset H \cup H^-$ and will be moved into $H$ and therefore $P'$ has a special $k$--simplex again.\\

It remains to prove the condition for meekness of $P$. Let $P$ be as before with $P = (\Sig \oplus Q) \cap \bigcap_{H \in \cH} H$ such that (a) --- (c) are satisfied and furthermore all $H$ separate $Q$ trivially. We may assume w.l.o.g. that all $H \in \cH$ are no bounding hyperplanes of $P$. Then $Q \subset \bigcap_{H \in \cH} H$ and thus $\cF(P) \bs \Sig$ is contained in the affine basis space of $Q$. Hence $P$ is meek. Notice that the intersection with hyperplanes of this type corresponds to the transformation of direct sums into pyramids via intersections which we already mentioned after Proposition \ref{SadefShell}.
\end{proof}

Note that the Theorem does not state that it is possible to move all vertices of $\Sig \oplus Q$ in the corresponding hyperplanes of $P$ \emph{successively} and get a polytope in $\cS_{(n,k)}(Q)$ after moving the vertices of one or some particular $\wh H^-$. It states that one gets a polytope in $\cS_{(n,k)}(Q)$ if one moves the vertices of $\Sig \oplus Q$ to all corresponding hyperplanes \textit{simultaneously} such that the hyperplanes satisfies the conditions (a) --- (c).

Instead of moving vertices of $\Sig \oplus Q$ into certain corresponding hyperplanes $P$, the whole process might of course also be interpreted as simply intersecting $\Sig \oplus Q$ with the corresponding hyperplanes. This is maybe the easier way to understand the connection between $P$ and $\Sig \oplus Q$ but it has the disadvantage that it is not a continuous deformation anymore.\\

Thus, the reason why it is hard to classify all wild polytopes with special $k$--simplex and basis polytope $Q$ is that it is difficult to figure out the sets of simultaneously $Q$--supportable hyperplanes satisfying the upper conditions if $Q$ is sufficiently complicated. The situation gets even worse if one is not only interested in the $f$--vector of the specific wild polytope but also in its face lattice. The following example shows that --- in contrast to the meek case ---the face lattice of a wild polytope with certain $f$--vector is not unique:
\begin{exa}
Both polytopes on the pictures below have the $f$--vector $(1,10,22,14,1)$ and both are wild polytopes with special 1--simplex and the 8--gon as basis polytope. Anyway, the polytopes have different face lattices and thus are non--isomorphic resp. have a different geometric realization.
\begin{figure}[H]
	\begin{center}
	\begin{picture}(400,200)(0,0)
		\put(0,0){\includegraphics[scale=0.35]{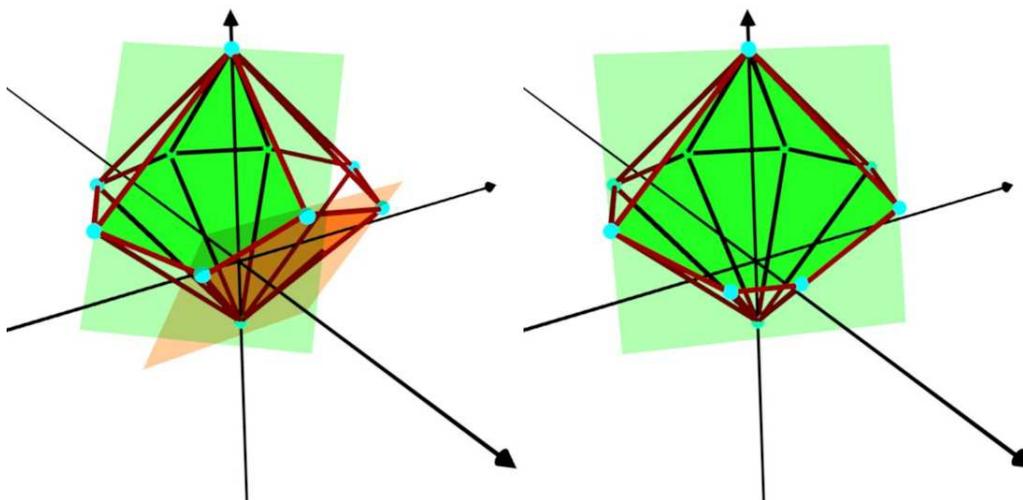}}
	\end{picture}
	\caption[Two non--isomorphic, combinatorial equivalent wild polytopes.]{Two non--isomorphic, combinatorial equivalent wild polytopes with special 1--simplex and identical basis polytope.}
	\label{AbbExaWildPoly}
	\end{center}
\end{figure}
\end{exa}

Indeed we are unable to give specific formulas for possible $f$--vectors of wild polytopes with special simplex without further information about the basis polytope. But we are able to find an upper bound for the $f$--vector:

\begin{cor}
Let $P \in \PW_{(n,k)}(Q)$ with special simplex $\Sig$ and $P' \in \PM_{(n,k)}(Q)$ the polytope one obtains by applying the vertex projection on the vertices not belonging to $\Sig$. I.e.: $P' = \conv(\phi(\cV(P) \bs \cV(\Sig)) \cup \cV(\Sig))$. Then we have for the $f$--vectors $f_P$ and $f_{P'}$:
\begin{eqnarray*}
		f_{P_0}	& =	& f_{P'_0} \\
		f_{P_i}	& <	& f_{P'_i} \quad \text{ for } i \in \{n-1,n\} \\
		f_{P_i}	& \leq	& f_{P'_i} \quad \text{ for } i \in \{1,\ldots,n-2\}.
\end{eqnarray*}
Observe that $f_{P'}$ is componentwisely bounded by $f_{\Sig \oplus Q}$ and $f_{\Sig \oplus Q}$ can always be computed explicitly.
\label{CorBoundaries}
\end{cor}

\begin{proof}
$f_{P_0} = f_{P'_0}$ follows directly from the definition of $\phi$. Theorem \ref{ThmWildPolyCharac} shows that the $P$ can be achieved from $P'$ by intersecting $P'$ with additional hyperplanes such that each intersection combines facets $G_1,\ldots,G_r$ of $P'$ to one large facet of $P$.
\end{proof}

Notice furthermore that all polynomials in $\PS_{(n,k)}(Q)$ have the same reverse lexicographic triangulation with vertexorder $\tau := (v_p,\ldots,v_0)$ and $\cV(\Sig) = (v_0,\ldots,v_k)$ since Athanasiadis' Lemma \ref{Lem2Atha1} tells us that for every polytope $P \in \PS_{(n,k)}(Q)$ the reverse lexicographic triangulation $\tri_{\tau}(P)$ is given by:
\begin{eqnarray*}
	\tri_{\tau}(P)	& =	& \tri_{\tau}(Q) * \Sig
\end{eqnarray*}

As we have already mentioned after Definition \ref{DefQSeperable}, $Q$--separability / --supportability becomes rather trivial if $\dim(Q) = 2$. Thus, we want to handle this case separately:
\begin{cor}
All Polytopes $P \in \cS_{n,n-2}(Q)$ can be constructed out of $\Sig_{n-2} \oplus Q$ by intersection with hyperplanes $H \in \cH$ defined on $\cV(\Sig_{n-2} \oplus Q)$ such that for every $H \in \cH$:
\begin{enumerate}
	\item[(a)] There exists exactly one $w_0 \in H^+ \cap \cV(\Sig_{n-2})$, $H$ contains exactly two $v_0,v_1 \in \cV(Q)$ and all $w_i \in \cV(\Sig) \bs \{w_0\}$ or it contains $Q$.
	\item[(b)] Every vertex $v \in \cV(Q)$ is contained in at most $n-2$ negative halfspaces of $H \in \cH$ or $Q$--containing hyperplanes.
\end{enumerate}
\label{CorWildPolyBasisPolyDim2}
\end{cor}

\begin{proof}
We have to show that the two upper conditions fall together with the three conditions of Theorem \ref{ThmWildPolyCharac} for the case $\dim(Q) = 2$. First we show that the conditions of the corollary imply the three conditions of the theorem. Observe that every single hyperplane defined in (a) is $Q$--supportable since $Q$ is an $n$--gon and thus simplicial and only consists of vertices and facets. Hence every vertex in $H^-$ can be moved into $H$ without changing the face lattice of $\cF(\partial Q)$. Condition (b) ensures that $v$ is still contained in $n$ facets after intersecting with all hyperplanes. In $\Sig_{n-2} \oplus Q$ every vertex $v$ is contained in $2(n-1)$ facets: $v$ is contained in 2 edges of $Q$, every edge forms $n-1$ facets of the direct sum induced by the $n-1$ subsets of $\cV(\Sig_{n-2})$ of cardinality $n-2$. Every intersections transforms two of these hyperplanes into a single big one. Thus, $v$ may be contained in at most $n-2$ negative halfspaces since otherwise $v$ would lie only in $n-1$ hyperplanes at all and thus be no vertex of $P$ anymore. This is condition (a) of Theorem \ref{ThmWildPolyCharac}. Conditions (b) and (c) are trivial.\\

The inverse way is easy since (b) follows from $Q$--supportability (with the argument from above), $w_0 \in H^+$ is trivial and every hyperplane may only contain $2$ or all vertices of $Q$ since it is an $n$--gon. If $H$ contains only $2$ vertices of $\cV(Q)$ we obviously need $n-2$ more to define a hyperplane and hence $\lf(\cV(\Sig) \bs \{w_0\}\ri) \subset H$
\end{proof}

Although one does not get \textit{one} combinatorial description of all wild polytopes with special simplex we want to close the article with showing that in principle one can get very explicit results for certain classes of wild polytopes with special simplex. We show that by investigating the $n$--cubes $\square_n$ which always come with a special 1--simplex due to the fact that they are weakly Hannar (cf. Example \ref{ExaWeaklyHannar}).

Since we have seen that the wild polytopes with special simplex and 2--dimensional basis polytopes are rather easy, we want to start with discussing the 4--cube $\square_4$ as a non--trivial example since it is a wild polytope with special 1--simplex and 3--dimensional basis polytope. The Schlegel--diagram of the 4--cube looks the following way (cf. eg. \cite[p. 136]{Zieg1}):

\begin{figure}[H]
	\begin{center}
	\begin{picture}(150,150)(0,0)
		\put(0,0){\includegraphics[scale=0.3]{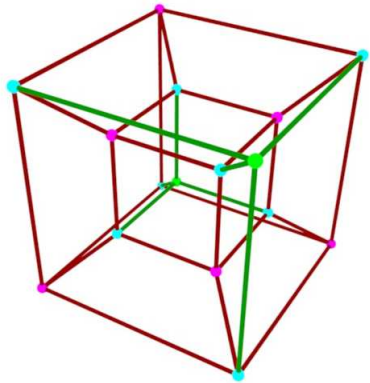}}
	\end{picture}
	\caption{The Schlegel--diagram of a 4--cube.}
	\label{Abb4CubeSchlegel}
	\end{center}
\end{figure}

The two green vertices form a special 1--simplex $\Sig$. Each turquoise vertex is connected to one vertex of $\Sig$; each pink vertex is not connected to a vertex of $\Sig$. Observe that there are seven other special 1--simplices due to the symmetry of the 4--cube. 

The complex $\cF(\square_4) \bs \cV(\Sig)$ consists of all faces of $\square_4$ not containing a green vertex. This is obviously a 2--dimensional polytopal complex since it isomorphic to the boundary complex of the 3--dimensional basis polytope. It consists of twelve 4--gons which are the orange planes in the following figure. We splitted it up in two pictures for clarity reasons. In the left picture we have three 4--gons in the outer 3--cube and three 4--gons connecting the outer with the inner cube. In the right picture we have three 4--gons in the inner cube and again three 4--gons connecting the outer with the inner cube.

\begin{figure}[H]
	\begin{center}
	\begin{picture}(300,150)(0,0)
		\put(0,0){\includegraphics[scale=0.3]{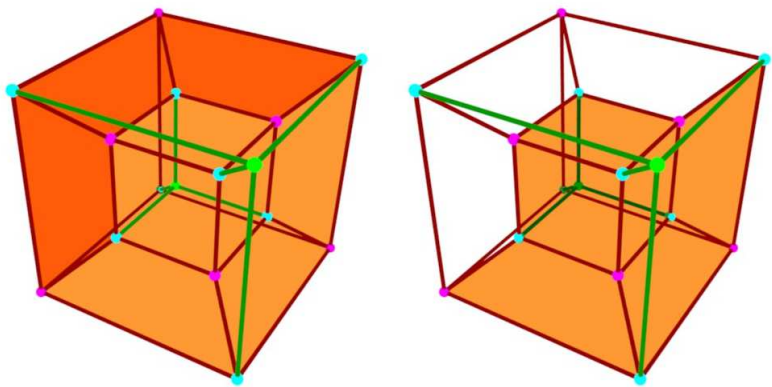}}
	\end{picture}
	\caption{The boundary complex of the basis polytope in the 4--cube.}
	\label{Abb4CubeBasisComplex}
	\end{center}
\end{figure}

If one realizes this polytopal complex as the boundary complex of a 3--polytope, the basis polytope of $\square_4$, it is easy to see that one gets a polytope which is combinatorially equivalent with the one on the left picture of figure \ref{Abb4CubeBasispoly} (cf. p. \pageref{Abb4CubeBasispoly}). We denote this polytope as $Q$. Thus, $\square_4$ is realized by intersecting $\Sig_1 \oplus Q$ with a couple of hyperplanes. Observe that $\Sig_1 \oplus Q$ is a bipyramid such that both of the pyramids have a Schlegel--diagram as in the middle picture of figure \ref{Abb4CubeBasispoly}. In fact, $\square_4$ is realized by intersecting each pyramid of the bipyramid $\Sig_2 \oplus Q$ with the four 3--spaces given in the right picture of figure \ref{Abb4CubeBasispoly} (respectively the other four 3--spaces defined by each 3--set of the other four turquoise vertices). To see that fill in all missing faces in the Schlegel--diagram of $\square_4$ to transform it into $\Sig_1 \oplus Q$ and observe that one has locally the situation on the left picture of the following figure which one has to transform in the right one to obtain one facet of $\square_4$:

\begin{figure}[H]
	\begin{center}
	\begin{picture}(300,100)(0,0)
		\put(0,0){\includegraphics[scale=0.4]{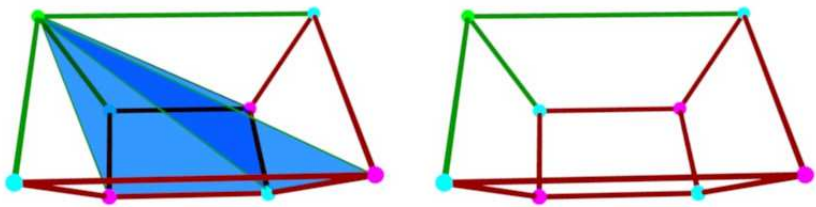}}
	\end{picture}
	\caption{Transformation of facets of a direct sum into a facet of the 4--cube.}
	\label{Abb4CubeFacet}
	\end{center}
\end{figure}

This is exactly what happens in the right picture of figure \ref{Abb4CubeBasispoly} and proves at the same time that the set of hyperplanes in the picture are simultaneously $Q$--separating as we have claimed it in Example \ref{ExaQSeperable}. 

$\square_4$ together with the $n$--cube in lower dimensions already provide the structure of the basis polytopes of $\square_n$:
\begin{prop}
The basis polytope $Q_n$ of $\square_n$ is the zonotope given by the  Minkowski sum of $\square_{(n-1)}$ and the vector $(1,\ldots,1) \in \fR^{n-1}$. Thus, it has the $f$--vector which is given by
\begin{eqnarray*}
	f_{(Q_n)}^{(0)}	& =	& f_{(\square_n)}^{(0)} + f_{(Q_{(n-1)}}^{(0)} \text{ and } \\
	f_{(Q_n)}^{(i)}	& =	& f_{(\square_n)}^{(i)} + f_{(Q_{(n-1)}}^{(i)} + f_{(Q_{(n-1)}}^{(i-1)}	\text{ for all } i \in \{1,\ldots,n-1\}
\end{eqnarray*}
\label{PropBasispolyCube}
\end{prop}

\begin{proof}
Let w.l.o.g. all $\square_n$ be centrally symmetric with vertices in $\{-1,1\}^n$. For every $\square_n$ we call the set of all faces which contain the point $(-1,\ldots,-1) \in \fR^n$ the \emph{lower} and the set of those which contain the point $(1,\ldots,1)$ the \emph{upper half} of $\square_n$. Observe that the intersection of the closure of the lower and the upper half is the basis polytope $Q_n$ since $\square_n$ is weakly Hannar and hence $(-1,\ldots,-1),(1,\ldots,1)$ form a special 1--simplex in $\square_n$.

Let w.l.o.g. $(-1,\ldots,-1,1),(1,\ldots,1,-1)$ be the vertices of the special simplex $\Sig_n$ we want to look at in $\square_n$. We investigate $\cF(\square_n) \bs \cV(\Sig_n)$: $\square_n$ is the prism over $\square_n$. Hence on the bottom $\cF(\square_n) \bs \cV(\Sig_n)$ consists of the closure of the lower half of $\square_{n-1}$, because $(1,\ldots,1,-1)$ is also one vertex of a special simplex $\Sig_{n-1}$ in $\square_{n-1}$. The same argument holds for the top with the upper half of $\square_{n-1}$. Between top and bottom $\cF(\square_n) \bs \cV(\Sig_n)$ contains all the faces $F \times (0,\ldots,0,2)$ where $F$ denotes the faces which are in the closure of the lower half of $\square_{n-1}$ on the bottom and $F$ shifted by $(0,\ldots,2)$ is contained in the closure of the upper half of $\square_{n-1}$ shifted by $(0,\ldots,2)$ in the top. These are exactly the faces belonging to $Q_{n-1}$.

This already gives the $f$--vector. Now we investigate the zonotope $\square_{n-1} + (1,\ldots,1)$: The Minkowski sum leaves the lower half of $\square_{n-1}$ invariant, turns every face $F$ of the intersection of the closure of the lower and the upper half (which is $Q_{n-1}$) into $F \times (1,\ldots,1)$ and shifts the upper half of $\square_{n-1}$ along $(1,\ldots,1)$ such that every face $F$ of the upper half, belonging to $Q_{n-1}$, is transfered to $F$ shifted by $(1,\ldots,1)$. Hence the face lattices of $\square_{n-1} + (1,\ldots,1)$ and $Q_n$ coincide.
\end{proof}

\vspace*{1cm}

\begin{scriptsize}
\noindent Timo de Wolff \\
Institute of Mathematics, Goethe--University Frankfurt am Main, Germany \\
email: \texttt{wolff(at)math.uni-frankfurt.de}\\
\url{http://www.uni-frankfurt.de/fb/fb12/mathematik/dm/personen/dewolff/}	
\end{scriptsize}

\vspace*{1cm}

\listoffigures
\begin{scriptsize} 
All figures were constructed with the open source 3D creator \emph{Blender}. See: \url{www.blender.org}
\end{scriptsize}

\vspace*{1cm}

\bibliographystyle{alpha}
\bibliography{./../References}
\end{document}